\documentclass{amsart}
\usepackage{textcomp}
\usepackage{amsmath}
\usepackage{mathrsfs}
\usepackage{amsthm}
\usepackage[linktocpage]{hyperref} 
\usepackage{amsfonts,graphics,amsthm,amsfonts,amscd,latexsym}
\usepackage{epsfig}
\usepackage{flafter}
\usepackage{mathtools}
\usepackage{comment} 
\usepackage{stmaryrd}
\usepackage{enumitem}
\usepackage{epigraph}
\usepackage{tikz-cd}
\hypersetup{
colorlinks=true,
linkcolor=red,
citecolor=green,
filecolor=blue,
urlcolor=blue
}
\usepackage{tikz}
\usetikzlibrary{graphs,positioning,arrows,shapes.misc,decorations.pathmorphing}

\tikzset{
>=stealth,
every picture/.style={thick},
graphs/every graph/.style={empty nodes},
}

\tikzstyle{vertex}=[
draw,
circle,
fill=black,
inner sep=1pt,
minimum width=5pt,
]
\usepackage[position=top]{subfig}
\usepackage[alphabetic]{amsrefs}
\usepackage{amssymb}
\usepackage{color}

\calclayout

\usetikzlibrary{decorations.pathmorphing}
\tikzstyle{printersafe}=[decoration={snake,amplitude=0pt}]

\newcommand{\ord}{\operatorname{ord}}

\newcommand{\lcm}{\operatorname{lcm}}

\newcommand{\oo}{\mathcal{O}}

\def\O#1.{\mathcal {O}_{#1}}
\def\pr #1.{\mathbb P^{#1}}
\def\af #1.{\mathbb A^{#1}}
\def\ses#1.#2.#3.{0\to #1\to #2\to #3 \to 0}
\def\xrar#1.{\xrightarrow{#1}}
\def\K#1.{K_{#1}}
\def\bA#1.{\mathbf{A}_{#1}}
\def\bM#1.{\mathbf{M}_{#1}}
\def\bL#1.{\mathbf{L}_{#1}}
\def\bB#1.{\mathbf{B}_{#1}}
\def\bK#1.{\mathbf{K}_{#1}}
\def\subs#1.{_{#1}}
\def\sups#1.{^{#1}}

\usepackage{tikz}
\usetikzlibrary{matrix,arrows,decorations.pathmorphing}

\newtheorem{theorem}{Theorem}[section]
\newtheorem{lemma}[theorem]{Lemma}
\newtheorem{proposition}[theorem]{Proposition}
\newtheorem{corollary}[theorem]{Corollary}

\theoremstyle{definition}
\newtheorem{definition}[theorem]{Definition}
\newtheorem{example}[theorem]{Example}

\newtheorem{question}[theorem]{Question}
\newtheorem{construction}[theorem]{Construction}

\newtheorem{remark}[theorem]{Remark}

\theoremstyle{remark}

\numberwithin{equation}{section}

\usepackage[all]{xy}

\newcounter{rownumber}[figure]
\newcounter{rownumber-irr}[figure]
\newcounter{rownumber-p1}[figure]
\begin{document}
\title[Jordan property for groups of self-maps of complex manifolds]{Jordan property for groups of bimeromorphic self-maps of complex manifolds with large Kodaira dimension}
\begin{abstract}
We prove that the image of the pluricanonical representation of a group of bimeromorphic automorphisms of a complex manifold  has bounded finite subgroups. As a consequence, we show that the group of bimeromorphic automorphisms of an $n$-dimensional complex manifold whose Kodaira dimension is at least $n-2$, satisfies the Jordan property.
\end{abstract}
\author[K.~Loginov]{Konstantin Loginov}
\address{Steklov Mathematical Institute of Russian Academy of Sciences, Moscow, Russia; Laboratory of Algebraic Geometry, HSE University, Russian Federation; Laboratory of AGHA, Moscow Institute of Physics and Technology}
\email{loginov@mi-ras.ru}

\maketitle
\section{Introduction}
In this note, we study the groups of bimeromorphic automorphisms of complex manifolds. More precisely, we concentrate on the finite subgroups of such groups. One way to study their structure is to establish the so called \emph{Jordan property}, that is, to show that any finite subgroup of a given group admits a normal abelian subgroup of bounded index.

In what follows, we denote by $\mathrm{Aut}(X)$ the group of automorphisms of a complex manifold $X$, by $\mathrm{Bim}(X)$ the group of bimeromorphic self-maps of $X$, and by $\kappa(X)$ the Kodaira dimension of~$X$. 

We list several results on the Jordan property, see Definition \ref{def-jordan-bfs}. 
If $X$ is a compact complex surface, then $\mathrm{Aut}(X)$ is Jordan, while $\mathrm{Bim}(X)$ is Jordan if and only if $X$ is not bimeromorphic to a product of $\mathbb{P}^1$ and an elliptic curve, {see \cite[Theorems 1.6 and 1.7]{PSh21}}. 
If $X$ is a projective variety over a field of characteristic zero, then the group $\mathrm{Aut}(X)$ is Jordan {\cite{MZh18}}.

For K\"ahler varieties, there are the following results. 
If $X$ is a compact K\"ahler variety, then the group $\mathrm{Aut}(X)$ is Jordan {\cite{Kim18}}.
If $X$ be a compact $3$-dimensional K\"ahler variety with $\kappa(X)\geq 0$ and $q(X)=\mathrm{H}^1(X, \oo_X)>0$, then the group $\mathrm{Bim}(X)$ is Jordan {\cite{PSh21b}}. Recently, A. Golota showed that it is possible to get rid of the assumption on the irregularity of $X$:

\begin{theorem}[{\cite{Go21}}]
\label{thm-gol}
Let $X$ be a $3$-dimensional compact K\"ahler variety with $\kappa(X)\geq~0$. Then the group $\mathrm{Bim}(X)$ is Jordan. 
\end{theorem}

Our goal is to derive an analogous result for complex varieties without the assumption that $X$ is K\"ahler. We prove the following

\begin{theorem}
\label{main-thm-2}
Let $X$ be an $n$-dimensional compact complex variety with $\kappa(X)\geq n-2$. Then the group $\mathrm{Bim}(X)$ is Jordan.
\end{theorem}

\begin{corollary}
\label{main-cor-3} 
Let $X$ be a $3$-dimensional compact complex variety with $\kappa(X)> 0$, then the group~$\mathrm{Bim}(X)$ is Jordan. 
\end{corollary}

Our proof is based on the study of bounded finite subgroup property for the pluricanonical representation of the group of bimeromorphic self-maps. More precisely, let us denote by $\rho$ the \emph{pluricanonical representation} of the group $\mathrm{Bim}(X)$, that is, the action of $\mathrm{Bim}(X)$ on the space of holomorphic $m$-forms. Also, by $\overline{\rho}(\mathrm{Bim}(X))$ we denote the projectivization of the pluricanonical representation which acts on image of the pluricanonical map, see section \ref{subsec-pluricanonical-rep} for details. For convenience, we call it the \emph{projective pluricanonical representation}. We say that a group has \emph{bounded finite subgroups}, if the cardinality of any finite subgroup of such group is bounded by a constant which is independent of the subgroup. Then we prove the following

\begin{theorem}
\label{prop-pluri-bfs}
Let $X$ be a compact complex manifold. Then both the image of the pluricanonical representation $\rho(\mathrm{Bim}(X))$ and the image of the projective pluricanonical representation $\overline{\rho}(\mathrm{Bim}(X))$ have bounded finite subgroups.
\end{theorem}

Combining Theorem \ref{thm-gol} and Theorem \ref{prop-pluri-bfs}, we deduce
\begin{corollary}
\label{main-cor-4}
Let $X$ be an $n$-dimensional compact K\"ahler variety with $\kappa(X)\geq n-3$. Then the group $\mathrm{Bim}(X)$ is Jordan.
\end{corollary}


In the proof of Theorem \ref{prop-pluri-bfs}, we follow the ideas used by Ueno in \cite[\S 14]{Ue75} where the following result, whose proof is attributed to Deligne, was established: 
\begin{theorem}[{\cite[Corollary 14.10]{Ue75}}]
\label{thm-finite-moishezon}
Let $X$ be a compact complex manifold. Assume that~$X$ is Moishezon. Then the group $\rho(\mathrm{Bim}(X))$ is finite.
\end{theorem}
In contrast, in Theorem \ref{prop-pluri-bfs} we do not need the assumption that $X$ is Moishezon, but instead of finiteness, we establish a weaker property of $\rho(\mathrm{Bim}(X))$. However, this is the best we can get as there exist compact complex manifolds with infinite $\rho(\mathrm{Bim}(X))$, see Example \ref{example-infinite}. We can also ask about finiteness of the group $\overline{\rho}(\mathrm{Bim}(X))$. If $X$ is a compact complex surface, then $\overline{\rho}(\mathrm{Bim}(X))$ is finite. Indeed, 
finiteness of $\overline{\rho}(\mathrm{Bim}(X))$ is trivial for $\kappa(X)=0$, 
for $\kappa(X)=1$ its finiteness is proven in \cite[Lemma 8.2]{PSh21}, 
and for $\kappa(X)=2$ the group $\mathrm{Bim}(X)$ is itself finite by \cite[Corollary~14.3]{Ue75}.  
In higher dimensions, we do not know about finiteness results for $\overline{\rho}(\mathrm{Bim}(X))$ when the Kodaira dimension is not maximal possible.  
Hence we formulate the following

\begin{question}
Does there exist a compact complex manifold $X$ such that the image of the projective pluricanonical representation $\overline{\rho}(\mathrm{Bim}(X))$ is infinite?
\end{question}

We refer to \cite{HX16} and references therein for finiteness results in the case of projective pairs.

\

\textbf{Acknowledgements.} 
The work was performed at the Steklov International Mathematical Center and supported by the Ministry of Science and Higher Education of the Russian Federation (agreement no. 075-15-2022-265), by the HSE University Basic Research Program, Russian Academic Excellence Project ``5-100'', and the Simons Foundation. He is Young Russian Mathematics award
winner and would like to thank its sponsors and jury. The author is grateful to Constantin Shramov for suggesting the problem, careful reading of the draft of the paper, and useful discussions. The author thanks  Andrey Ryabichev, Andrey Soldatenkov, Aleksei Golota and Alexander Kuznetsov for helpful conversations. 

\section{Preliminaries}
We work over the field of complex numbers. We refer the reader to \cite{Ue75} for the basic facts and definitions. 

We recall some of them. By a \emph{complex variety} we mean an irreducible reduced complex space. 
 A smooth complex variety is called a \emph{complex manifold}. A \emph{complex surface} is a complex manifold of dimension~$2$. By a (Zariski) open subset of complex manifold $X$ we mean a subset of the form~$X \setminus Z$, where $Z$ is a closed analytic subset in $X$. By a \emph{typical point} of $X$ we mean a point in some non-empty Zariski open subset of $X$.

A proper surjective holomorphic map $f\colon X\to Y$ of reduced complex spaces is called a \emph{modification} if there exist  nowhere dense closed analytic subsets $V\subset X$ and $W\subset Y$ 
such that $f$ restricts to a biholomorphic map $X\setminus V\to Y\setminus W$. A \emph{meromorphic map} $f\colon X\dashrightarrow Y$ of reduced complex analytic spaces $X$ and $Y$ is a holomorphic map defined outside a nowhere dense subset such that the closure of its graph~$\overline{\Gamma_f}\subset X\times Y$ is a closed analytic subset of $X\times Y$, and the natural projection~$\overline{\Gamma_f} \to X$ is a modification. A meromorphic map $f$ is called \emph{bimeromorphic} if the natural projection $\overline{\Gamma_f}\to Y$ is a modification as well. 

Given a reduced complex space $X$, by $\mathrm{Bim}(X)$ we denote its group of bimeromorphic self-maps, and by $\mathrm{Aut}(X)$ we denote its group of biholomorphic self-maps.

\begin{example}
A bimeromorphic self-map of an open subset $U$ in a complex variety~$X$ may not induce a bimeromorphic self-map of $X$. Indeed, by \cite[Section 2]{VT87} there exist complex manifolds $X_1$ and $X_2$ such that they are two non-bimeromorphic compactifications of $(\mathbb{C}^\times)^2$. Put $(\mathbb{C}^\times)^2=U_1\subset X_1$, $(\mathbb{C}^\times)^2=U_2\subset X_2$, and 
\[
X = X_1 \times X_2 \supset U = U_1\times U_2.
\]
Define a biholomorphic map $f$ on $U=U_1\times U_2$ by the formula $f(x_1, x_2)=(x_2, x_1)$ when $x_i\in U_i$ for $i=1,2$, and $U_1$ is identified with $U_2$. Assume that $f$ induces a bimeromomorphic map on $X$. Consider the closure of its graph
\[
\overline{\Gamma_f} \subset X\times X = X_1\times X_2 \times X_1\times X_2.
\]
Then both projections 
\[
p_i=\mathrm{pr}_i|_{\overline{\Gamma_f}}\colon \overline{\Gamma_f}\to X=X_1\times X_2
\] 
are modifications. 
Consider the subset $U_1\times \{q\}\subset X_1\times \{q\}$ where $q\in U_2$. We have $f(U_1\times \{q\})=\{q\}\times U_2$. Hence 
\[
p_1^{-1}(U_1\times \{q\}) = U_1\times \{q\}\times \{q\}\times U_1\subset  \overline{\Gamma_f},
\]
Put
\[
\Gamma_g={X_1\times \{q\}\times \{q\}\times X_2} \cap \overline{\Gamma_f}.
\]
Consider $\Gamma_g$ as a closed analytic subset that defines a graph of a bimeromorphic map $g\colon X_1\dashrightarrow X_2$. Indeed, this follows from the fact that the proper maps $\Gamma_g\to X_i$ are modifications for $i=1,2$. This contradiction shows that $f$ does not induce a bimeromorphic map on $X$.
\end{example}

\subsection{Pluricanonical representation}
\label{subsec-pluricanonical-rep} 
Let $X$ be a complex manifold. Assume that for the Kodaira dimension of $X$ we have $\kappa(X)\geq 0$. For any $m\geq 1$, consider the map
\[
\sigma_m\colon X \dasharrow \mathbb{P}(\mathrm{H}^0(X, \oo_X(mK_X))^\vee).
\]
For a sufficiently big and divisible $m$, we have $\dim \sigma_m(X)=\kappa(X)$ and the fibers of $\sigma_m$ are connected. For such $m$, we put $\sigma=\sigma_m$. Passing to a modification of $X$, we may assume that $\sigma$ is a morphism. Obviously, this does not change the group $\mathrm{Bim}(X)$. Then a typical fiber of $\sigma$ is smooth, and hence irreducible (cf. \cite[Theorem 5.10]{Ue75}). 
We have the following 

\begin{proposition}[{see e.g. \cite[Lemma 6.3]{Ue75}}]
The group $\mathrm{Bim}(X)$ acts on $\sigma(X)$ biholomorphically.
\end{proposition}

Put $V=\mathrm{H}^0(X, \oo_X(mK_X))$ and consider the induced homomorphism
\begin{equation}
\label{pluricanonical-rep}
\rho = \rho_m\colon \mathrm{Bim}(X)\to \mathrm{GL}(V),
\end{equation}
which is called a \emph{pluricanonical representation} of $\mathrm{Bim}(X)$. 
Consider the natural exact sequence of groups
\[
1 \to \mathbb{C}^\times \to \mathrm{GL}(V) \xrightarrow{p} \mathrm{PGL}(V)\to 1.
\]
Put $\overline{\rho}=p\circ \rho$. Then by the \emph{projective pluricanonical representation} of $\mathrm{Bim}(X)$ we mean the image~$\overline{\rho}(\mathrm{Bim}(X))$. Observe that the action of $\overline{\rho}(\mathrm{Bim}(X))$ on $\sigma(X)$ is faithful. 


\subsection{Group theory} We start with the following definition.
\begin{definition}
\label{def-jordan-bfs}
We say that a group $\Gamma$ has \emph{bounded finite subgroups}, if there is a constant~$C=~C(\Gamma)$ such that for any finite subgroup $G\subset \Gamma$ we have $|G|\leq C$.

Following \cite[Definition 2.1]{Po09}, we say that a group $\Gamma$ is \emph{Jordan} (or has \emph{Jordan property}) if there is a constant $J = J(\Gamma)$ such that every finite subgroup $G \subset \Gamma$ contains a normal abelian subgroup $A \subset G$ of index at most $J$.
\end{definition}

Note that the Jordan property (as well as the property to have bounded finite subgroups) behaves badly with respect to passing to the quotient group. Indeed, any group $G$ which does not enjoy this property can be realised as a quotient of a free group on some number of generators, and the free group clearly satisfies Jordan property (and has bounded finite subgroups) since the only finite subgroup in it is the trivial group. See also Example \ref{ex-quotient-bfs} in the case when $G$ is a subgroup of the general linear group. Now, let us make the following straightforward observation.
\begin{remark}
\label{lem-iff-jordan}
Suppose that the following sequence of groups is exact
\[
1\to G_1 \to G \to G_2 \to 1.
\]
Assume that $G_2$ has bounded finite subgroups. Then $G$ is Jordan if and only if $G_1$ is Jordan.
\end{remark}

\begin{example}
\label{thm-gl-bfs}
A classical result of C. Jordan says that for every positive integer $n$ the group~$\mathrm{GL}_n(\mathbb{C})$ is Jordan, see e.g. \cite[Theorem 36.13]{CR62}. As a consequence, any linear algebraic groups over an arbitrary field of characteristic zero is Jordan. 

The group $\mathrm{GL}_n(\mathbb{\mathbb{K}})$ has bounded finite subgroups for any field $\mathbb{K}$ which is finitely generated over~$\mathbb{Q}$. This theorem was proven by H. Minkowski, {see e.g. \cite[Theorem 5]{Ser07} and~\cite[\S 4.3]{Ser07}}. It follows that $\mathrm{GL}_n(\mathbb{Z})$ has bounded finite subgroups.
\end{example}

The next result can be seen as an effective version of Burnside theorem for subgroups of the general linear group. 
\begin{theorem}[{\cite[Theorem 1]{HP76}}]
\label{thm-effective-burnside}
Let $n$ and $d$ be positive integers. 
Let $\Gamma \subset \mathrm{GL}_n(\mathbb{C})$ be a subgroup. Suppose that for every $g \in \Gamma$ one has~$g^d = 1$. Then $\Gamma$ is finite and $|\Gamma|$ divides $d^n$.
\end{theorem}

\begin{corollary}
\label{cor-pgl-effective-burnside}
Let $\Gamma \subset \mathrm{PGL}_n(\mathbb{C})$ be a subgroup. Suppose that for every $g \in \Gamma$ one has~$g^d = 1$. Then $\Gamma$ is finite and $|\Gamma|$ divides $(nd)^n$.
\end{corollary}
\begin{proof}
Write the following exact sequence
\[
0\to \mathbb{Z}/n\mathbb{Z} \to \mathrm{SL}(V) \xrightarrow{p'}  \mathrm{PGL}(V) \to 1.
\]
Consider the group $\widetilde{\Gamma}=p'^{-1}(\Gamma)$. Pick any element $\widetilde{g}\in \widetilde{\Gamma}$ and put $g=p'(\widetilde{g})$. Note that  
$\ord(\widetilde{g})$ divides $n\ord(g)$ which divides $nd$ by assumption. By Theorem~\ref{thm-effective-burnside}, we have that $\widetilde{\Gamma}$ is finite and $|\widetilde{\Gamma}|$ divides $(nd)^n$. Hence $\Gamma$ is finite and $|\Gamma|$ divides $(nd)^n$ as well.
\end{proof}

Now we establish some boundedness properties of subgroups of the general linear group that we will use later.

\begin{lemma}
\label{lem-bounded-element-order}
Let $V$ be an $n$-dimensional vector space over $\mathbb{C}$. Let $g\in\mathrm{GL}(V)$ be an element of finite order $d$. Assume that there exists a constant $C$ such that for any eigenvalue $\alpha$ of $g$ the degree of the field extension $[\mathbb{Q}(\alpha):\mathbb{Q}]$ is bounded by $C$. Then~
\[
d\leq 2^{n}C^{2n}.
\] 
\end{lemma}
\begin{proof}
Let $\alpha_1, \ldots, \alpha_n$ be the eigenvalues of $g$. Then each $\alpha_i$ is a root of unity whose degree we denote by $d_i$, so that we have $\alpha_i^{d_i}=1$. Then $d=\lcm(d_i)_{1\leq i\leq n}$.
It is well known that 
\[
[\mathbb{Q}(\alpha_i):\mathbb{Q}]=\varphi(d_i)\] 
where $\varphi(d_i)$ is the Euler's totient function. One easily checks that $d_i\leq 2\varphi(d_i)^2$, so we estimate 
\begin{equation*}
\label{eq-estimate-1}
\sqrt{d_i/2} \leq \varphi(d_i) = [\mathbb{Q}(\alpha_i):\mathbb{Q}] \leq C.
\end{equation*} 
 
Consequently,  
\begin{equation*}
\label{eq-estimate-2}
d = \lcm(d_i)_{1\leq i\leq n}\leq d_1\cdot\ldots\cdot d_n \leq 
2^{n}C^{2n},
\end{equation*}
which completes the proof.
\end{proof}

\begin{lemma}
\label{lem-bfs-image-bfs}
Let $V$ be an $n$-dimensional vector space over $\mathbb{C}$, and let $\Gamma \subset \mathrm{GL}(V)$ be a subgroup. Consider an exact sequence of groups
\[
1 \to \mathbb{C}^\times \to \mathrm{GL}(V) \xrightarrow{p} \mathrm{PGL}(V)\to 1.
\]
Assume that there exists a constant $C$ such that for any eigenvalue $\alpha$ of $g$ the degree of the field extension $[\mathbb{Q}(\alpha):\mathbb{Q}]$ is bounded by $C$. 
Then both $\Gamma$ and $p(\Gamma)$ have bounded finite subgroups.  
\end{lemma}
\begin{proof}
First we show that $\Gamma$ has bounded finite subgroups. Consider a finite subgroup $H\subset \Gamma$. 
By Lemma \ref{lem-bounded-element-order}, given any element $g\in H$, for its order $d=\ord (g)$ we have 
\[
d\leq 2^{n}C^{2n}.
\] 
Using Theorem~\ref{thm-effective-burnside}, we conclude that $|H|$ is bounded by $2^{n^2}C^{2n^2}$. Hence $\Gamma$ has bounded finite subgroups.

Now we prove that $p(\Gamma)$ has bounded finite subgroups. 
Consider a finite subgroup $G\subset p(\Gamma)$. 
By Corollary \ref{cor-pgl-effective-burnside}, it is enough to bound the order $d$ of an arbitrary element $g\in G$ by a constant which is independent of $g$. 
Pick any element $\widetilde{g}\in \mathrm{GL}(V)$ in $p^{-1}(g)\cap \Gamma$. Note that~$\widetilde{g}$ may have infinite order. Denote by $\alpha_1, \ldots, \alpha_n$ the eigenvalues of $\widetilde{g}$. Since $g^d = \mathrm{id}$, we have $\widetilde{g}^d=\lambda\cdot \mathrm{id}$ for some $\lambda\in \mathbb{C}^\times$. In particular, $\alpha_i^d=\lambda$ for any $1\leq i\leq n$. Thus, the element $\alpha_1^{-1}\widetilde{g}$ satisfies 
\[
\ord(\alpha_1^{-1}\widetilde{g})=d.
\] 
Note that its eigenvalues are $1$ and $\alpha_i/\alpha_1$ for $2\leq i\leq n$. Each $\alpha_i/\alpha_1$ belongs to the field $\mathbb{Q}(\alpha_1, \alpha_i)$, and the degree of the field extension $[\mathbb{Q}(\alpha_1, \alpha_i):\mathbb{Q}]$ is bounded by $C^2$. Applying Lemma \ref{lem-bounded-element-order} to the element $\alpha_0^{-1}\widetilde{g}$, we see that \[d\leq 2^{n}C^{4n}.\] Using Corollary \ref{cor-pgl-effective-burnside}, we conclude that $|G|$ is bounded by $2^{n^2}C^{4n^2}n^n$. Hence $p(\Gamma)$ has bounded finite subgroups.
\end{proof}

The following example due to C. Shramov shows that in general if $\Gamma\subset \mathrm{GL}(V)$ has bounded finite subgroups, it might be not true that $\Gamma' = p(\Gamma)\subset \mathrm{PGL}(V)$ enjoys the same property.

\begin{example}
\label{ex-quotient-bfs}
Let $V$ be a $2$-dimensional vector space over $\mathbb{C}$. By $U$ we denote the set of all complex roots of unity. Start with a group 
\[
\Gamma' = \left\{ 
\begin{pmatrix}
\xi & 0 \\
0 & 1  
\end{pmatrix}
\Big|\ \, \xi \in U\, \right\}\subset \mathrm{PGL}(V).
\]
Observe that $\Gamma' \simeq \mathbb{Q}/\mathbb{Z}$. Now, we define a subgroup $\Gamma\subset \mathrm{GL}(V)$ with $p(\Gamma)=\Gamma'$. Choose a countable subset in $\mathbb{C}$ consisting of elements $\{\zeta(\xi)\}_{\xi\in U}$, indexed by all roots of unity $\xi\in U$, that are algebraically independent over~$\mathbb{Q}$. Then we define
\[
\Gamma = \left\{\,
\zeta(\xi)\cdot 
\begin{pmatrix}
\xi & 0 \\
0 & 1  
\end{pmatrix}
\Big|\ \, \xi\in U\, \right\}\subset \mathrm{GL}(V).
\]
It is clear that $\Gamma$ is a free Abelian group, and hence $\Gamma$ has bounded finite subgroups, while $\Gamma'$ has finite subgroups of arbitrarily large order.
\end{example}

\section{Resolution of singularities}
In this section, we construct some special ``fiberwise'' resolution of singularities which we will use in the proof of Theorem \ref{prop-pluri-bfs}. Similar construction for Moishezon varieties was proposed by Ueno in the proof of \cite[Theorem 14.10]{Ue75}. However, since we work in the analytic category, instead of his argument involving scheme-theoretic generic point, we propose another purely analytic construction.

\begin{definition}
By an \emph{embedded smooth model} of a possibly reducible and non-reduced complex subspace~$Y'$ of a complex manifold $M'$ we mean a possibly reducible smooth complex subspace $Y$ of a complex manifold $M$ together with a commutative digram  
\[
\begin{tikzcd}
M \ar[r, "h"] & M' \\
Y \ar[r, "g"] \ar[u, hook] & Y' \ar[u, hook] 
\end{tikzcd}
\] 
where $h$ and $g$ are proper surjective holomorphic maps, the map $h$ is bimeromorphic, and $Y$ is the strict transform of $Y'$ via the map $g$.
\end{definition}

The last condition in the above definition means that for each irreducible component $F$ of~$Y$ there exists an irreducible component $F'$ of $Y'$ considered with the reduced structure such that the map $g|_{F}\colon F\to F'$ is bimeromorphic, and the map $g$ induces a bijection between the irreducible components of $Y$ and the irreducible components of $Y'$. In particular, each component of $Y$ is a strict transform of the corresponding irreducible component of $Y$ via the map $g$.

Note that an embedded smooth model always exists by the usual embedded resolution of singularities \cite[Theorem 2.0.2]{Wl08}.


In what follows, we call a map $\pi\colon Y\to W$ \emph{equidimensional over} $U\subset W$, if all the components of all its fibers over $U$ have the same dimension. We say that $\pi$ is \emph{equidimensional}, if it is equidimensional over $W$.

\begin{lemma}
\label{lem-open-subset}
Consider the following commutative diagram
\begin{equation}
\begin{tikzcd}
W \times M' \ar[rd, "\mathrm{pr}_1"] & \\
Y' \ar[u, hook] \ar[r, "\pi"] & W 
\end{tikzcd}
\end{equation}
where 
\begin{itemize}
\item
$Y'$ is a reduced, possibly reducible compact analytic subspace of $W\times M'$, 
\item
$M'$ is a smooth (possibly non-compact) complex manifold, 
\item
$W$ is an irreducible reduced (possibly non-compact) complex space, 
\item
$\pi$ is a proper surjective equidimensional map. 
\end{itemize}

Then there exists an open subset $U\subset W$ and a commutative diagram 
\begin{equation}
\begin{tikzcd}
\widetilde{M_U} \ar[r, "h_U"] & U \times M' \ar[rd, "\mathrm{pr}_1"] & \\
\widetilde{\pi^{-1}(U)} \ar[r, "g_U"] \ar[u, hook] & \pi^{-1}(U)\ar[u, hook] \ar[r, "\pi"] & U \subset W
\end{tikzcd}
\end{equation}
such that 
\begin{itemize}
\item
$\widetilde{\pi^{-1}(U)}\subset \widetilde{M_U}$ is an embedded smooth model of $\pi^{-1}(U)\subset U \times M'$,
\item
the map $\pi \circ g_{U}$ 
is smooth,
\item
for any point $\phi\in U$ we have that $Y_\phi\subset M_\phi$ is an embedded smooth model of $Y'_\phi\subset M'$: 
\begin{equation}
\label{diag-phi}
\begin{tikzcd}
M_\phi \ar[r, "h_\phi"] & M' \ar[rd] & \\
Y_\phi \ar[r, "g_\phi"] \ar[u, hook] & Y'_\phi\ar[u, hook] \ar[r, "\pi"] & \phi\in U
\end{tikzcd}
\end{equation}
where $Y_\phi = (\pi \circ g_{U})^{-1}(\phi)$, $Y'_\phi=\pi^{-1}(\phi)$ and $M_\phi=(\mathrm{pr}_1\circ h_U)^{-1}(\phi)$.
\end{itemize}
\end{lemma}
\begin{proof}
Passing to an open subset of $W$, if needed, we may assume that $W$ is smooth. By \cite[Theorem 2.0.2]{Wl08}, there exists a resolution of singularities of $Y_0=Y'$, 
which is realized as a sequence of blow-ups $g_1, \ldots, g_t$ in smooth centers $Z_i\subset M_i$ starting from $M_0=W\times M'$ as shown in the next diagram
\begin{equation}
\label{diag-global}
\begin{tikzcd}
\ & Z_{t-1} \ar[d, hook'] & \ldots & Z_1 \ar[d, hook'] & Z_0 \ar[d, hook'] & \\
M=M_t \ar[r, "h_t"] & M_{t-1} \ar[r, "h_{t-1}"] & \ldots \ar[r, "h_2"] & M_1 \ar[r, "h_1"] & M_0 = W\times M' \ar[rd] & \\
Y=Y_t \ar[r, "g_t"] \ar[u, hook] & Y_{t-1} \ar[r, "g_{t-1}"] \ar[u, hook] & \ldots \ar[r, "g_2"] & Y_1 \ar[r, "g_{1}"] \ar[u, hook] & Y_0 = Y' \ar[u, hook] \ar[r, "\pi"] & W 
\end{tikzcd}
\end{equation}
Here $h_i$ for $1\leq i\leq t$ is the blow-up of $M_i$ in a smooth center $Z_{i-1}\subset Y_{i-1}$, each $Y_i$ is the strict transform of $Y_0$, and $g_i=h_i|_{Y_i}$; here by the strict transform we mean the union of strict transforms of all the irreducible components of $Y_0$. 
Moreover, by \cite[Theorem 2.0.2]{Wl08} we may assume that each center $Z_i$ is disjoint from the set $\mathrm{Reg}(Y_i)$ of smooth points of $Y_i$.
Put 
\[
h=h_1\circ\ldots\circ h_t, \quad \quad \quad \quad g=g_1\circ\ldots\circ g_t, \quad \quad \quad \quad \pi_i = \pi\circ g_1\circ\ldots\circ g_i \quad \text{for}\ 0\leq i\leq t.
\] 

To start with, we define an open subset $U\subset W$ by the condition that the map $\pi_t = \pi \circ g$ is smooth over~$U$. 
Then $W\setminus U$ is a proper closed analytic subset in $W$, cf. \cite[Corollary 1.7]{Ue75}.

For an arbitrary point $\phi\in U$, we consider the diagram
\begin{equation}
\label{diag-phi-big}
\begin{tikzcd}
\ & Z_{t-1}\cap Y_{t-1,\phi} \ar[d, hook'] & \ldots & Z_1\cap Y_{1,\phi} \ar[d, hook'] & Z_0\cap Y_\phi \ar[d, hook'] \\
M_\phi=M_{t,\phi} \ar[r, "h_{t,\phi}"] & M_{t-1,\phi} \ar[r, "h_{t-1,\phi}"] & \ldots \ar[r, "h_{2,\phi}"] & M_{1,\phi} \ar[r, "h_{1,\phi}"] & M_{0,\phi} = M' \ar[rd] \\
Y_\phi = Y_{t,\phi} \ar[r, "g_{t,\phi}"] \ar[u, hook] & Y_{t-1,\phi} \ar[r, "g_{t-1,\phi}"] \ar[u, hook] & \ldots \ar[r, "g_{2,\phi}"] & Y_{1,\phi} \ar[r, "g_{1,\phi}"] \ar[u, hook] & Y_{0,\phi}=Y'_\phi \ar[r, "\pi"] \ar[u, hook] & \phi\in U 
\end{tikzcd}
\end{equation}
where $Y_{i, \phi}=\pi_i^{-1}(\phi)$ and $g_{i, \phi}=g_i|_{Y_{i,\phi}}$ for $1\leq i\leq t$. Put $g_\phi=g|_{Y_\phi}$.

We claim that, possibly after shrinking $U$ at each of $t$ steps in diagram \eqref{diag-global}, each map~$\pi_i$ is equidimensional over $U$. Indeed, $\pi_0=\pi$ is equidimensional by assumption. By induction, for a fixed $i$ with $1\leq i\leq t$ we may assume that $\pi_j$ is equidimensional for all~$j\leq i-1$. We show that, possibly after shrinking $U$, the map $\pi_{i}$ is equidimensional. Consider the exceptional set $E_i=\mathrm{Exc}(g_i)$ which maps to $Z_{i-1}$. By \cite[Theorem V.4.10]{BS76}, the subset of points in~$U$ over which the maps $\pi_i$ and $\pi_i|_{E_i}$ are not flat, is a proper closed analytic subset of $U$. Hence, possibly after shrinking~$U$, we see that these maps are flat over $U$. By \cite[Proposition V.2.4]{BS76} the maps $\pi_i$ and $\pi_i|_{E_i}$ are equidimensional over $U$. In particular, $Z_i$ either maps to $U$ surjectively, or $\pi_i(Z_i)\cap U=\emptyset$. (In the latter case, $h_i$ induces an isomorphism of $\pi_{i}^{-1}(U)$ with $\pi_{i-1}^{-1}(U)$, hence we may skip this step in our construction).

Consequently, for any $i$ with $1\leq i\leq t$ and any $\phi\in U$, we can write 
\[
\dim Y' - 1 \geq \dim E_i = \dim U + \dim (E_i\cap Y_{i, \phi}).
\] 
In particular, if $E_i$ is a divisor on $Y_i$, then the above inequality becomes an equality. 
On the other hand,  
\[
\dim Y' = \dim U + \dim Y'_{\phi} = \dim U + \dim Y_{i, \phi}.
\] 
We conclude that
\begin{equation}
\label{eq-dim}
\dim (E_i\cap Y_{i, \phi}) \leq \dim Y_{i, \phi} - 1.
\end{equation}
Hence, for any $1\leq i\leq t$ the intersection $E_i\cap Y_{i, \phi}$ is a proper closed subset of $Y_{i,\phi}$ which has codimension at least $1$ in any irreducible component of $Y_{i, \phi}$. 

As a consequence, for each $1\leq i\leq t$ we have that $Y_{i, \phi}$ is the strict transform of $Y_{i-1, \phi}$ via the map $g_{i, \phi}$. Hence $Y_{\phi}$ is the strict transform of $Y'_{\phi}$ via the map $g_{\phi}$. 
Observe that~$Y_\phi$ is smooth (possibly, reducible), since by construction the map $\pi_t|_U$ is smooth. Hence~$Y_\phi\subset M_\phi$ is an embedded smooth model of $Y'_\phi\subset M'$ for any $\phi\in U$.
Put 
\[
\widetilde{M_U}=(h\circ \mathrm{pr_1})^{-1}(U), \quad \quad \quad \widetilde{\pi^{-1}(U)} = \pi_t^{-1}(U), \quad \quad \quad h_U = h|_{\widetilde{M_U}}, \quad \quad \quad g_{U}=g|_{\pi_t^{-1}(U)}.
\] 
This completes the proof.
\end{proof}

\begin{lemma}
\label{lem-bounded-betti}
In the assumptions and notation of Lemma \ref{lem-open-subset}, if the complex space $W$ is compact then for any point $\phi\in W$ we can find an embedded smooth model $Y_\phi\subset M_\phi$ of $Y'_\phi=\pi^{-1}(\phi)\subset M'_\phi$ as in the next diagram
\begin{equation}
\label{diag-phi}
\begin{tikzcd}
M_\phi \ar[r, "h_\phi"] & M' \ar[rd] & \\
Y_\phi \ar[r, "g_\phi"] \ar[u, hook] & Y'_\phi\ar[u, hook] \ar[r, "\pi"] & \phi\in U
\end{tikzcd}
\end{equation}
such that the Betti numbers of $Y_\phi$ are bounded uniformly in $\phi\in W$.

\end{lemma}
\begin{proof}
By Lemma \ref{lem-open-subset}, there exists an open subset $U\subset W$ and a resolution of singularities of 
\[
g_{U}\colon \widetilde{\pi^{-1}(U)}\to \pi^{-1}(U)
\] 
such that it gives an embedded smooth model $g_\phi\colon Y_\phi\to Y'_\phi$ for any $\phi\in U$, where $Y_\phi$ is embedded into $M_\phi$ as in diagram \ref{diag-phi} (see also diagram \ref{diag-phi-big}). Put $U_0=U$. Consider the following decomposition of the complement of $U_0$ in $W$ into the union of irreducible components: 
\[
W\setminus U_0=\bigcup_{j=1}^{k_1} W_{j}, \quad \quad \quad \quad \quad Y'\setminus \pi^{-1}(U_0) = \bigcup_{j=1}^{k_1} Y'_{j}, \quad \quad \quad \quad \quad \pi_{j} = \pi|_{Y'_{j}}\colon Y'_{j}\to W_{j}.
\]
Applying Lemma \ref{lem-open-subset} to each $Y'_{j}$ and  $\pi_{j}\colon Y'_{j}\to W_{j}$ and iterating this process, we obtain a decomposition of $W$ into finitely many locally closed analytic subsets $U_j$
\[
W = \bigcup_{j=0}^k U_j, \quad \quad \quad \quad \quad \quad \quad \quad \quad Y' = \bigcup_{j=0}^k Y'_{j}, \quad \quad \quad \quad \quad \quad \quad \quad \quad Y'_{j} = \pi^{-1}(U_{j})
\]
and resolutions of singularities $g_{j}\colon Y_{j}\to Y'_{j}$
such that each fiber $Y_\phi$ of the induced map 
\[
\pi\circ g_j\colon Y_{j}\to U_{j}
\] 
over $\phi\in U_{j}$ is a smooth model of $Y'_\phi$, and $Y_\phi$ is embedded into $M_\phi$ as in diagram \ref{diag-phi}. By construction, each map 
$
\pi\circ g_j
$
is smooth, so by Ehresmann's theorem it is differentially locally trivial. Hence, we see that for each $n$ the Betti numbers 
$
b_n(Y_\phi)
$ 
are constant for any $\phi\in U_j$. Since there are only finitely many $U_j$, we conclude that the Betti numbers of $Y_\phi$
are uniformly bounded for each $\phi\in W$. This completes the proof. 
\end{proof}

\section{Proof of main results}
We use the notation of section \ref{subsec-pluricanonical-rep}. 
First we prove Theorem \ref{prop-pluri-bfs} in the simple case when the linear system $|K_X|$ defines the pluricanonical map. 

Recall that for a compact smooth (possibly, reducible) equidimensional complex space  $X$ one can define the canonical class $K_X$ and the vector spaces of the holomorphic forms $\mathrm{H}^0(X, \oo_X(K_X))$ in the usual way. Also, the group of bimeromorphic self-maps $\mathrm{Bim}(X)$ is also well-defined in this situation.  

\begin{proposition}
\label{main-thm-m-1}
Let $X$ be a compact smooth (possibly, reducible) equidimensional complex space. Let $g\in \mathrm{Bim}(X)$, and let $\beta$ be an eigenvalue of $g$ considered as an operator acting on 
\[
\mathrm{H}^0(X, \oo_{X}(K_X)).
\]
Then the degree of the field extension $[\mathbb{Q}(\beta):\mathbb{Q}]$ is bounded by the Betti number 
\[
b_n(X)=\mathrm{rk} (\mathrm{H}^n(X, \mathbb{Z})/\mathrm{tors}).
\]
where $n$ is the dimension of $X$. 
\end{proposition}
\begin{proof}
Consider the map 
\begin{equation}
\label{inclusion}
\mathrm{H}^0(X, \oo_X(K_X))\xhookrightarrow{} \mathrm{H}^n(X, \mathbb{C}) = (\mathrm{H}^n(X, \mathbb{Z})/\mathrm{tors})\otimes \mathbb{C}
\end{equation}
which is injective by de Rham theorem. It is well-known that any element $g\in\mathrm{Bim}(X)$ acts as a linear operator on $\mathrm{H}^n(X, \mathbb{Z})/\text{tors}$, see e.g. {\cite[14.3]{Ue75}}. 
Since the natural action of $g\in\mathrm{Bim}(X)$ on the space of holomorphic forms is compatible with its action on $\mathrm{H}^n(X, \mathbb{Z})/\mathrm{tors}$, we conclude that~$\beta$ is an algebraic integer, and $[\mathbb{Q}(\beta):\mathbb{Q}]$ is bounded by the Betti number $b_n(X)$.
\end{proof}

\begin{corollary}
Let $X$ be a compact complex manifold of dimension $n$. If the linear system~$|K_X|$ defines the pluricanonical map, then both $\rho(\mathrm{Bim}(X))$ and $\overline{\rho}(\mathrm{Bim}(X))$ have bounded finite subgroups.
\end{corollary}
\begin{proof}
Let $g\in \mathrm{Bim}(X)$, and let $\beta$ be an eigenvalue of $g$ considered as a linear operator acting on~
$\mathrm{H}^0(X, \oo_{Y}(K_X))$. 
By Proposition \ref{main-thm-m-1}, the degree of the field extension $[\mathbb{Q}(\beta):\mathbb{Q}]$ is bounded by the Betti number
$
b_n(X)
$. Thus from Lemma~\ref{lem-bfs-image-bfs} it follows that both $\rho(\mathrm{Bim}(X))$ and $\overline{\rho}(\mathrm{Bim}(X))$ have bounded finite subgroups.
\end{proof}

Now we are going to prove Theorem \ref{prop-pluri-bfs} in the general case. 

\begin{construction}
\label{constr-cover}
We follow \cite[14.4]{Ue75}. Let $X$ be a complex manifold, and let $M'$ be the total space of the line bundle~$\oo_X(K_X)$ considered as a non-compact complex manifold. Let~$\tau \colon M'\to X$ be the natural projection. Then the \emph{$m$-canonical cover} of $X$ associated with a holomorphic $m$-form 
\[
\phi\in\mathrm{H}^0(X, \oo_X(mK_X))
\] 
is a compact complex subspace $Y'_\phi$ of $M'$ constructed as follows.


Let $\{ U_i \}_{i\in I}$ be an open covering of $X$ with local coordinates $z_i^1, \ldots, z_i^n$ on $U_i$. Locally on~$U_i$ we can write 
\[
\phi = \phi_i(z_i^1, \ldots, z_i^n) (dz_i^1\wedge \ldots \wedge dz_i^n)^{\otimes m},
\] 
where $\phi_i\in \oo_X(U_i)$. 
Note that $M'$ admits an open covering $\{ V_i \}_{i\in I}$ where $V_i=U_i\times \mathbb{C}$ with local coordinates $z_i^1, \ldots, z_i^n, w_i$. 
For each $i\in I$, define a closed analytic subset $Y'_\phi\subset M'$ locally on $V_i$ by the equations 
\begin{equation}
\label{eq-local-definition}
w_i^m = \phi_i (z_i^1, \ldots, z_i^n).
\end{equation}
\end{construction}

\begin{remark}
Note that the complex space $Y'_\phi$ constructed above may be not connected, reducible and even non-reduced. However, by construction it is equidimensional.
\end{remark}

For the reader's convenience, we reproduce the proof of the following lemma. 

\begin{lemma}[{\cite[14.4]{Ue75}}]
\label{lem-local-bounding}
Let $X$ be a compact complex manifold of dimension $n$, and $m$ be a number such that the linear system $|mK_X|$ defines the pluricanonical map. 
Let 
\[
\phi\in \mathrm{H}^0(X, \oo_X(mK_X))
\] 
be an eigenvector of an element $g\in \mathrm{Bim}(X)$ acting on $\mathrm{H}^0(X, \oo_X(mK_X))$ with an eigenvalue $\alpha$. Then the degree of the field extension $[\mathbb{Q}(\alpha):\mathbb{Q}]$ is bounded by $C=C(X, \phi)$ depending only on $X$ and $\phi$, but not on $g$.

More precisely, we can take $C=b_n(Y_\phi, \mathbb{Z})$ where $Y_\phi\subset M_\phi$ is any embedded smooth model of the $m$-canonical cover $Y'_\phi\subset M'$.

\end{lemma}
\begin{proof}

Let $f\colon Y'_\phi \to X$ be the $m$-canonical cover of $X$ associated with an $m$-form $\phi$ as in Construction \ref{constr-cover}. Note that $Y'_\phi$ is embedded into the total space $M'$ of the line bundle~$\oo_X(K_X)$. Define a holomorphic $n$-form $\psi$ on $M'$ in the following way. Let $\{ U_i \}_{i\in I}$ be a sufficiently fine open covering of $X$ with local coordinates $z_i^1, \ldots, z_i^n$ on $U_i$. Then $M'$ admits an open covering $\{ V_i \}_{i\in I}$ where $V_i=U_i\times \mathbb{C}$ with local coordinates $z_i^1, \ldots, z_i^n, w_i$. For each $i\in I$, in the chart $V_i$ put
\[
\psi = w_i dz_i^1\wedge \ldots \wedge dz_i^n.
\] 
The holomorphic forms in the right-hand side agree on the intersections $V_i\cap V_j$, hence they define a global homomorphic $n$-form $\psi$ on $M'$.

Note that any bimeromorphic map $g\in \mathrm{Bim}(X)$ induces a bimeromorphic map $g_{M'}\in \mathrm{Bim}(M')$. More precisely, if $g(U_i)\cap U_j\neq \emptyset$, then the map 
\[
g_{M'}|_{V'_j}\colon V'_j = g^{-1}(g(U_i)\cap U_j)\times\mathbb{C}\to U_j\times \mathbb{C}
\] 
is given by 
\[
(z_i^1, \ldots, z_i^n, w_i) \mapsto \left(g^1(z_i), \ldots, g^n(z_i), \det \left(\frac{\partial (g^1(z_i), \ldots, g^n(z_i))}{\partial(z_i^1, \ldots, z_i^n)}\right)^{-1} w_i\right).
\]

Define a holomorphic automorphism $\mu_\beta\in \mathrm{Aut}(M')$ locally given by 
\[
\mu_\beta\colon(z_i^1, \ldots, z_i^n, w_i) \mapsto (z_i^1, \ldots, z_i^n, \beta w_i),
\]
where $\beta\in \mathbb{C}$ satisfies $\beta^m=\alpha$. Since by assumption we have $g^*\phi=\alpha\phi$, it follows that 
\[
\phi_i(g^1(z_i),\ldots,g^n(z_i)) = \alpha \phi_i(z_i^1,\ldots, z_i^n) \det \left(\frac{\partial (g^1(z_i), \ldots, g^n(z_i))}{\partial(z_i^1, \ldots, z_i^n)}\right)^{-m}.
\]
Thus, from \eqref{eq-local-definition} in Construction \ref{constr-cover} it follows that the map $h'=\mu_\beta\circ g_{M'}$ preserves $Y'_\phi$, and hence~$h'$ induces a bimeromorphic automorphism of $Y'_\phi$. 

Consider any embedded smooth model $Y_\phi$ of $Y'_\phi$ (we know that it exists by \cite[Theorem~2.0.2]{Wl08}), where $Y'_\phi$ is embedded into $M'$, and $Y_\phi$ is embedded into a smooth manifold $M$.
We denote the corresponding map as $f_2\colon Y_\phi\to Y'_\phi$. 
Let $f_1\colon Y'_\phi\to X$ be the restriction of the natural projection $\tau\colon M'\to X$ to $Y'_\phi\subset M'$. Finally, put 
\[
f=f_1\circ f_2\colon Y_\phi\to X.
\] 

Denote by $\omega$ a holomorphic form on $Y_\phi$ which is obtained by pulling back $\psi$ to $M$ and then restricting to $Y_\phi$. 
We have
\[
\omega^m = f^*\phi.
\]
Recall that $g^*\phi=\alpha\phi$. Note that the bimeromorphic map $h'=\mu_\beta\circ g_{M'}$ 
induces a bimeromorphic map $h$ on $M$. Since by the above $h'$ restricts to a bimeromorphic automorphism on $Y'_\phi$, it follows that $h$ restricts to a bimeromorphic automorphism on $Y_\phi$. Then
\[
(\mu_\beta\circ g_{M'})^* \psi = \beta\psi, \quad \quad h^*(\omega)=\beta \omega.
\]
Hence $\beta$ is an eigenvalue of an operator $h$ acting on 
$
\mathrm{H}^0(Y_\phi, \oo_{Y_\phi}(K_{Y_\phi})).
$
Then by Proposition~\ref{main-thm-m-1}, the degree of the field extensions $[\mathbb{Q}(\beta):\mathbb{Q}]$ is bounded by the Betti number $b_{n}(Y_\phi)$, and by construction $Y_{\phi}$ depends only on $X$ and~$\phi$. Since $\beta^m=\alpha$, 
the degree of the field extension $[\mathbb{Q}(\alpha):\mathbb{Q}]$ is also bounded by~$b_{n}(Y_\phi)$.
\end{proof}

\begin{construction}
\label{constr-cover-global}
Let $X$ be a compact complex manifold. Following \cite[14.10]{Ue75}, we construct its \emph{global $m$-canonical cover} $Y'$ together with a map $f\colon Y'\to X$. It can be seen as a global version of the $m$-canonical cover \ref{constr-cover}. The construction is as follows. 

Let $\phi_0, \ldots, \phi_N$ be a basis of the vector space $\mathrm{H}^0(X, \oo_X(mK_X))$, and let 
\[
\mathbb{P}^N = \mathbb{P}(\mathrm{H}^0(X, \oo_X(mK_X))).
\] 
be the corresponding projective space. Consider a product $\mathbb{P}^N\times M'$ where $M'$ is the total space of $\oo_X(K_X)$ considered as a non-compact complex manifold. Let $\{ U_i \}_{i\in I}$ be a sufficiently fine open covering of $X$ with local coordinates $z_i^1, \ldots, z_i^n$ on~$U_i$. Let 
\[
u_0,\ldots, u_{k-1}, u_{k+1}, \ldots, u_N
\]
be the local coordinates on some affine chart $V_k$ of $\mathbb{P}^N$ for $0\leq k\leq N$. Locally on $U_i$ for any $0\leq j\leq N$ we can write 
\[
\phi_j = \phi_{j,i}(z^i_1,\ldots, z^i_n) dz^i_1\wedge\ldots\wedge dz^i_n.
\]
We define $Y'$ as a possibly singular analytic subspace of $\mathbb{P}^N\times M'$ locally in $V_k\times U_i\times \mathbb{C}$ defined by the equations
\[
w_i^m = u_1 \phi_{0, i} + \ldots + u_{k-1} \phi_{k-1, i} + \phi_{k, i} + u_{k+1} \phi_{k+1, i}  + \ldots + u_n \phi_{N, i},
\]
where $w_i$ is the coordinate along the factor $\mathbb{C}$. These equations agree on the intersections of charts and hence they define an analytic subspace $Y'$ of $M'$.
\end{construction}

\begin{remark}
Consider the natural projection 
\[
\pi \colon Y'\to \mathbb{P}^N,
\] 
induced by the projection map $\mathrm{pr}_1\colon \mathbb{P}^N\times M'\to\mathbb{P}^N$. Note that for any $\phi \in \mathrm{H}^0(X, \oo_X(mK_X))$, if we consider the corresponding point $\phi\in \mathbb{P}^N$, we have $\pi^{-1}(\phi)\simeq Y'_{\phi}$ where $Y'_\phi$ is the $m$-canonical cover associated with an $m$-form $\phi$ as in Construction \ref{constr-cover}. Conversely, any $m$-canonical cover $Y'_{\phi}$ associated with an $m$-form $\phi$ is isomorphic to $\pi^{-1}(\phi)$ for $\phi\in \mathbb{P}^N$. Note that the map $\pi$ is proper, and by construction its fibers are equidimensional. 
\end{remark}

We are ready to prove Theorem \ref{prop-pluri-bfs}. Note that unlike in Theorem \ref{thm-finite-moishezon}, we do not need the assumption that $X$ is Moishezon.

\begin{proof}[Proof of Theorem \ref{prop-pluri-bfs}]
Let $X$ be a compact complex variety of dimension $n$. Passing to a resolution of singularities if needed, we may assume that $X$ is smooth. Let $m$ be a number such that the linear system $|mK_X|$ defines the pluricanonical map. 

Put $V = \mathrm{H}^0(X, \oo_X(mK_X))$. We are going to prove that the groups $\Gamma=\rho(\mathrm{Bim}(X))\subset \mathrm{GL}(V)$ and $\Gamma'=\overline{\rho}(\mathrm{Bim}(X))\subset \mathrm{PGL}(V)$ have bounded finite subgroups. By Lemma \ref{lem-bfs-image-bfs}, it is enough to show that for any $g\in \Gamma$ and for any eigenvalue $\alpha$ of $g$ the degree of the field extension $[\mathbb{Q}(\alpha):\mathbb{Q}]$ is bounded by some constant which is independent of $g$ and $\alpha$. 

By Lemma \ref{lem-local-bounding}, we know that $[\mathbb{Q}(\alpha):\mathbb{Q}]$ is bounded by the Betti number $b_{n}({Y_{\phi}})$ where $Y_\phi$ is any embedded smooth model of the $m$-canonical cover $Y'_\phi$ associated to $\phi$ as in Construction \ref{constr-cover}. Hence it is enough to find embedded smooth models of $Y'_\phi$ such that the Betti numbers 
$
b_{n}({Y_{\phi}})
$ 
are uniformly bounded for any 
\[
\phi \in \mathrm{H}^0(X, \oo_X(mK_X)).
\] 
To this end, we apply Lemma \ref{lem-bounded-betti} to the global $m$-canonical cover $f\colon Y'\to X$ as in Construction~\ref{constr-cover-global}. More precisely, in the notation of Lemma \ref{lem-bounded-betti}, put 
$
\mathbb{P}^N=W 
$ and $\pi\colon Y'\to W$ to conclude that the Betti numbers of some embedded smooth models $Y_\phi\subset M_\phi$ of $Y'_\phi\subset M'$ are bounded for any $\phi$. 
This finishes the proof. 
\end{proof}

The following example shows that we may not have finite $\rho(\mathrm{Bim}(X))$ when $X$ is a compact complex compact manifold.

\begin{example}
\label{example-infinite}
By \cite{MS08}, there exists a non-algebraic K3 surface $X$ 
that admits a non-symplectic automorphism of infinite order acting on a holomorphic two-form via multiplication by a complex number of infinite order. Consequently, $\rho(\mathrm{Bim}(X))$ is not a finite group. Note also that by \cite[Remark 14.6]{Ue75}, there exists a $3$-dimensional complex torus such that the group $\rho(\mathrm{Bim}(X))$ is infinite.
\end{example}

To prove Theorem \ref{main-thm-2}, we will use the following technical results.

\begin{lemma}[{\cite[Lemma 2.9]{PSh21b}}]
\label{lem-fiber-non-negative-kodaira-dim}
Let $X$ and $Y$ be compact complex varieties and let $\sigma\colon X \dasharrow~Y$ be a dominant
meromorphic map. Suppose that $\kappa(X)\geq 0$. Let $F$ be a typical fiber of $\varphi$, and let $F'$ be an
irreducible component of $F$. Then $\kappa(F')\geq 0$.
\end{lemma}


\begin{lemma}[{\cite[Lemma 3.2]{PSh21b}}]
\label{lem-bfs+jordan}
Let $X$ and $Y$ be compact complex varieties and let $\sigma\colon X\dasharrow Y$ be a dominant meromorphic map. Let $F$ be a typical fiber of $\varphi$. Suppose that for any irreducible component $F'$ of $F$ the group $\mathrm{Bim}(F')$ is Jordan. Assume that the map $\varphi$ is equivariant with respect to the group $\mathrm{Bim}(X)$, and that the image of $\mathrm{Bim}(X)$ in $\mathrm{Bim}(Y)$ has bounded finite subgroups. Then the group $\mathrm{Bim}(X)$ is Jordan. 
\end{lemma} 

Now, we are ready to prove Theorem \ref{main-thm-2}.


\begin{proof}[Proof of Theorem \ref{main-thm-2}]
Let $X$ be a compact complex variety of dimension $n$ with $\kappa(X)\geq n-2$. Passing to a resolution of singularities if needed, we may assume that $X$ is smooth. 
Let $F$ be a typical fiber of the pluricanonical map. Then by section \ref{subsec-pluricanonical-rep} we may assume that $F$ is irreducible and smooth. 
By Lemma \ref{lem-fiber-non-negative-kodaira-dim}, we have $\kappa(F)\geq 0$. By assumption, we have $\dim F\leq 2$. Consider two cases. 
The first possibility is that $\dim F\leq 1$, and hence $\mathrm{Bim}(F) = \mathrm{Aut}(F)$ is obviously Jordan. The second possibility is that $\dim F=2$, and since $\kappa(F)\geq 0$, it follows that~$F$ is not bimeromorphic to a product of $\mathbb{P}^1$ and some other curve. Hence by \cite[Theorem 1.7]{PSh21} the group $\mathrm{Bim}(F)$ is Jordan. Since by Theorem \ref{prop-pluri-bfs} the group $\rho(\mathrm{Bim}(X))$ has bounded finite subgroups, applying Lemma \ref{lem-bfs+jordan} we conclude that $\mathrm{Bim}(X)$ is Jordan.
\end{proof}

Then, Corollary \ref{main-cor-3} follows immediately.

\begin{proof}[Proof of Corollary \ref{main-cor-4}]
Let $X$ be a compact complex variety of dimension $n$ with $\kappa(X)\geq n-3$. Passing to a resolution of singularities if needed, we may assume that $X$ is smooth. 
By Theorem~\ref{main-thm-2}, it is enough to consider the case $\kappa(X)=n-3$. By Lemma~\ref{lem-fiber-non-negative-kodaira-dim} a typical fiber $F$ is a $3$-dimensional K\"ahler manifold, and by Theorem \ref{thm-gol} we have that $\mathrm{Bim}(F)$ is Jordan. By Theorem \ref{prop-pluri-bfs} the image of projective pluricanonical representation $\overline{\rho}(\mathrm{Bim}(X))$ has bounded finite subgroups. Applying Lemma \ref{lem-bfs+jordan} to the pluricanonical map, we conclude that $\mathrm{Bim}(X)$ is Jordan.
\end{proof}






\begin{thebibliography}{lll}

\bibitem[BS76]{BS76}
C. B\u{a}nic\u{a}, O. St\u{a}n\u{a}\c{s}il\u{a},
\newblock {\em Algebraic methods in the global theory of complex spaces}.
\newblock Editura Academiei, Bucharest; John Wiley \& Sons, London-New York-Sydney, 296 pp, 1976.

\bibitem[CR62]{CR62}
C. W. Curtis and I. Reiner.
\newblock {\em Representation theory of finite groups and associative algebras}.
\newblock Pure and Applied Mathematics, Vol. XI. Interscience Publishers, a division of John Wiley \& Sons, New York-London, 1962.

\bibitem[Go21]{Go21}
A. Golota.
\newblock {\em Jordan property for groups of bimeromorphic automorphisms of compact K\"ahler threefolds}.
\newblock arXiv:2112.02673, 2021.

\bibitem[HX16]{HX16}
C. Hacon, C. Xu. 
\newblock {\em On finiteness of B-representations and semi-log canonical abundance}.
\newblock Minimal models and extremal Rays (Kyoto, 2011), 361-378, Adv. Stud. in Pure Math 70, Math. Soc. Japan, Tokyo, 2016.

\bibitem[HP76]{HP76}
M. Herzog, Ch. E. Praeger. 
\newblock {\em Boundedness for finite subgroups of linear algebraic groups}.
\newblock On the order of linear groups of fixed finite exponent. J. Algebra 43, 216--220, 1976.

\bibitem[Kim18]{Kim18}
J. H. Kim.
\newblock {\em Jordan property and automorphism groups of normal compact K\"ahler varieties}.
\newblock Commun. Contemp. Math. 20, no. 3, 1750024, 9 pp, 2018.
 
\bibitem[MS08]{MS08}
E. Macr\'i, P. Stellari.
\newblock {\em Automorphisms and autoequivalences of generic analytic surfaces}.
\newblock Journal of Geometry and Physics, vol. 58, 1, 133--164, 2008.
 
\bibitem[MZh18]{MZh18}
S. Meng and D.-Q. Zhang.
\newblock {\em Jordan property for non-linear algebraic groups and projective varieties}.
\newblock Amer. J. Math. 140, no. 4, 1133--1145, 2018.


\bibitem[Po09]{Po09}
V. L. Popov.
\newblock {\em On the Makar-Limanov, Derksen invariants, and finite automorphism groups of algebraic varieties}.
\newblock In Peter Russell's Festschrift, Proceedings of the conference on Affine Algebraic Geometry held in Professor Russell's honour, 1--5 June 2009, McGill Univ., Montreal., vol. 54 of Centre de Recherches Math\'ematiques CRM Proc. and Lect. Notes, pages 289--311, 2011.

\bibitem[PSh21]{PSh21}
Yu. Prokhorov, C. Shramov. 
\newblock {\em Automorphism groups of compact complex surfaces}.
\newblock Int. Math. Res. Notices, 2021(14):10490--10520, 2021.

\bibitem[PSh21b]{PSh21b}
Yu. Prokhorov, C. Shramov. 
\newblock {\em Finite groups of bimeromorphic selfmaps of non-uniruled K\"ahler threefolds}.
\newblock arXiv:2110.05825, 2021. To appear in Sb. Math.

\bibitem[Ser07]{Ser07}
J.-P. Serre.
\newblock {\em Bounds for the orders of the finite subgroups of G(k)}.
\newblock In Group representation theory, pages 405--450. EPFL Press, Lausanne, 2007.


\bibitem[Ue75]{Ue75}
K. Ueno.
\newblock {\em Classification theory of algebraic varieties and compact complex spaces}.
\newblock Lecture Notes in Mathematics, Vol. 439. Springer-Verlag, Berlin-New York, 1975. Notes written in collaboration with P. Cherenack.

\bibitem[VT87]{VT87}
V. Van Tan.
\newblock {\em On the compactification of strongly pseudoconvex surfaces III}.
\newblock Mathematische Zeitschrift Vol. 195, pages 259--267, 1987.

\bibitem[Wl08]{Wl08}
Ja. Wlodarczyk.
\newblock {\em Resolution of singularities of analytic spaces}.
\newblock Proceedings of 15th G\"okova Geometry-Topology Conference, 31--63, 2008.

\end{thebibliography}
\end{document}